\documentclass{amsart}
\usepackage{amsmath,amsfonts,amssymb,amsthm,epsfig,amscd,comment}

\newtheorem{Theorem}{Theorem}[section]
\newtheorem{Proposition}[Theorem]{Proposition} 
\newtheorem{Lemma}[Theorem]{Lemma}

\theoremstyle{definition}

\newtheorem{Remark}[Theorem]{Remark}

\newcommand{\p}{{\mathbb{P}^1}}
\newcommand{\ctimes}{\widetilde{\times}}
\newcommand{\gr}{\phantom{}_0\Gr^\ul}
\newcommand{\grt}{\phantom{}_0\widetilde{\Gr}^\ul}
\newcommand{\Vp}{\phantom{}_0 V}
\newcommand{\co}{\mathrm{Conf}}
\newcommand{\fl}{\mathcal{F}}
\newcommand{\tg}{\widetilde{\mathfrak{gl}}_N}
\newcommand{\supp}{\operatorname{supp}}
\newcommand{\C}{\mathbb{C}}
\newcommand{\spn}{\operatorname{Span}}
\newcommand{\Gr}{\mathrm{Gr}}
\newcommand{\K}{\mathcal{K}}
\newcommand{\edge}{\longrightarrow}
\renewcommand{\O}{\mathcal O}
\newcommand{\ul}{{\underline{\lambda}}}
\newcommand{\um}{{\underline{\mu}}}

\newcommand{\lam}{\lambda}
\newcommand{\Z}{\mathbb{Z}}
\newcommand{\Xnreg}{X^n_{\text{reg}}}
\newcommand{\Cnreg}{\C^n_{\text{reg}}}
\newcommand{\Pnreg}{\p^n_{\text{reg}}}

\begin{document}

\title[The BD Grassmannian and knot homology]{The Beilinson-Drinfeld Grassmannian and symplectic knot homology}

\author{Joel Kamnitzer}
\address{University of Toronto, Toronto, Canada}
\date{\today}

\begin{abstract}
Seidel-Smith and Manolescu constructed knot homology theories using symplectic fibrations whose total spaces were certain varieties of matrices.  These knot homology theories were associated to $SL(n) $ and tensor products of the standard and dual representations.  In this paper, we place their geometric setups in a natural, general framework.   For any complex reductive group and any sequence of minuscule dominant weights, we construct a fibration of affine varieties over a configuration space.  The middle cohomology of these varieties is isomorphic to the space of invariants in the corresponding tensor product of representations.  Our construction uses the Beilinson-Drinfeld Grassmannian and the geometric Satake correspondence.
\end{abstract}
\maketitle

\section{Introduction}
Let $ G $ be a complex reductive group.  For any $ n $-tuple $\ul =(\lambda_1, \dots, \lambda_n) $  of dominant weights of $ G$, consider the space of invariants $ \phantom{}_0 V^\ul := (V^{\lambda_1} \otimes \dots \otimes V^{\lambda_n})^G $.  This representation carries an action of the symmetric group by permuting the tensor factors.  More precisely, the group which acts is the stabilizer $ \Sigma_\ul $ of $ \ul $ under the action of the symmetric group $ \Sigma_n $ (so that if all $ \lambda_i $ are equal, then $ \Sigma_\ul = \Sigma_n $).

These representations of symmetric groups are of interest from the point of view of knot invariants.  Via quantum groups and R-matrices, these representations can be deformed to representations of braid groups.  This construction leads to knot invariants, such as the Jones polynomial.  Roughly speaking, the invariant of a knot $ K $ is the trace of a braid (in such a representation) whose closure is $K$.  

Khovanov \cite{K} has proposed categorifying these representations.  This means finding a triangulation category $ \phantom{}_0 \mathcal{D}^\ul $, with an action of an appropriate braid group, such that the Grothendieck group of $ \phantom{}_0 \mathcal{D}^\ul $ is isomorphic to $ \Vp^\ul $ with the above action of $ \Sigma_\ul $.  From these categorifications, Khovanov explained how to obtain more refined knot invariants by computing an appropriate categorical ``trace''.  (Actually Khovanov has proposed categorifying the braid group representations coming from quantum groups.  However, in this paper, we will just focus on categorifying the symmetric group representations.  Restricting to the categorification of these symmetric group representations is sufficient to get non-trivial braid group representations and knot invariants.)

The construction of these categorifications has thus far proceeded in a rather ``ad-hoc'' manner in different special cases, most notably when $ G = SL_2 $ or $ SL_m $.

One approach was taken by Seidel-Smith \cite{SS} and then extended by Manolescu \cite{M}.  They considered the cases when $ G = SL_2 $ \cite{SS} or $ G = SL_m $ \cite{M} and where $ \ul = (\omega_1, \dots, \omega_1, \omega_{m-1}, \dots, \omega_{m-1}) $, a tensor product of standard and dual representations.  In this case, they constructed a symplectic fibration $ S \rightarrow \C^n_\mathrm{reg} / \Sigma_\ul $.  Here $ S $ is a certain variety of $ nm \times nm $ matrices.  They proved that there is a well-defined (up to Hamiltonian isotopy) action of the braid group $ \pi_1(\C^n_\mathrm{reg} / \Sigma_\ul) $ on Lagrangian submanifolds in a fixed fibre.  The braid group acts by parallel transport through the fibration.  

The Lagrangians are objects in the derived Fukaya category.  So philosophically, this means that they constructed an action of the braid group on the derived Fukaya category and this category can be considered as $ \phantom{}_0 \mathcal{D}^\ul $ (recently, this perspective has been pursued by Reza Rezazadegan \cite{R}).  As supporting evidence, the middle cohomology of a fibre is isomorphic to $ \Vp^\ul $ as a representation of $ \Sigma_\ul $ (see \cite[section 3.1]{M}).  

The purpose of this paper is to place the geometric setup of Seidel-Smith and Manolescu in a more general and more natural framework.  For any complex reductive group $ G $ and any sequence of minuscule dominant weights $ \ul $, we construct a fibration of smooth complex affine varieties (and hence symplectic manifolds) $  \gr_{\co_\ul} \rightarrow \co_\ul $, where $ \co_\ul = \Pnreg /\Sigma_\ul $ is a coloured configuration space of points on $ \p $.  

The family $ \gr_{\co_\ul} \rightarrow \co_\ul $ is a type of Beilinson-Drinfeld Grassmannian \cite{BD}.  It is a moduli space of Hecke modifications of principal $G^\vee$ bundles on $ \p$ , where $ G^\vee $ denotes the Langlands dual group.  The fibre over a point $ x = [x_1, \dots, x_n] $ is defined by
\begin{equation*}
\begin{aligned}
\gr_x = \Big\{ (P, &\phi) : P \text{ is a principal } G^\vee \text{ bundle on } \p, \\
& \phi : P_0|_{X \smallsetminus \{x_1, \dots, x_n \}} \rightarrow P|_{X \smallsetminus \{x_1, \dots, x_n \}} \text{ is an isomorphism, }\\
& \text{$\phi$ is of Hecke type  $ \lam_i $ at $ x_i $, for each $ i $,} \\
& \text{and $ P $ is trivial.} \Big\}
\end{aligned}
\end{equation*}

We prove the following facts about this fibration.
\begin{enumerate}
\item The action (by monodromy) of $ \pi_1(\co_\ul) $ on the cohomology of a fibre factors through the group $ \Sigma_\ul$.  Moreover there is a $\Sigma_\ul $ equivariant isomorphism $ H^{\mathrm{mid}}(\gr_x) \cong \phantom{}_0 V^\ul $  (Proposition \ref{th:monod2}).
\item When two points come together in the base corresponding to dual weights, there is a local statement entirely analogous to lemmas in \cite{M} and \cite{SS} (Lemma \ref{th:semilocal}).

\item In the case when $ G = SL_m $ and $ (\lambda_1, \dots, \lambda_n) = (1, \dots, 1, m-1, \dots, m-1) $, then the portion of the fibration lying over points in $ \C $ is isomorphic to the fibration studied by Manolescu, which in turn reduces to the fibration studied by Seidel-Smith when $ m=2 $ (Theorem \ref{th:comp}).
\end{enumerate}

We prove statement (i) in section 2 as a consequence of the geometric Satake correspondence of Mirkovic-Vilonen \cite{MVi}.  In section 3, we prove statement (ii) as a consquence of the factorization property of the Beilinson-Drinfeld Grassmannian.  It is interesting to see how this key technical lemma of \cite{M} and \cite{SS} follows extremely naturally and easily in this setting.  In section 4, we prove statement (iii) following ideas of Mirkovic-Vybornov \cite{MVy} and Ngo \cite{N}.

In a future work, we hope to use this setup to define an action of the braid group on the Fukaya category of the fibres (categorifying the action of $ \Sigma_\ul $ on the middle cohomology) and then to construct homological knot invariants, following the approach of Seidel-Smith and Manolescu.

There is a close connection between this paper and the algebraic geometry approach to knot homology pursued jointly with Sabin Cautis in \cite{CK}.  We expect that the two constructions are related by hyperK\"ahler rotation.  The hyperK\"ahler structure of the varieties $\gr_x $ is described by Kapustin-Witten in sections 10.2 and 10.3 of \cite{KM}.  

\subsection*{Acknowledgements} 
Despite its short length, producing this paper required consultations with many  mathematicians.  In particular, I would like to thank Denis Auroux, Roman Bezrukavnikov, Alexander Braverman, Sabin Cautis, Edward Frenkel, Dennis Gaitsgory, Reimundo Heluani, Anton Kapustin, Ciprian Manolescu, Carl Mautner, Paul Seidel, Ivan Smith, Constantin Teleman, Edward Witten, Christopher Woodward, and Xinwen Zhu for helpful discussions.  During the course of this work, I was supported by a fellowship from the American Institute of Mathematics and by an NSERC Discovery Grant.

\section{The fibration}
We begin by reviewing different versions of the Beilinson-Drinfeld Grassmannian and exploring their properties using the geometric Satake correspondence.  These varieties were introduced by Beilinson-Drinfeld in \cite{BD}.

\subsection{The affine Grassmannian} Let $ G $ be a complex reductive group and let $ G^\vee $ be its Langlands dual group.  Let $ \Lambda $ denote the set of weights of $ G $ which is the same as the set of coweights of $ G^\vee $.  Let $ \Lambda_+ $ denote the subset of dominant weights.

Let $ \K = \C((t)) $ and $ \O = \C[[t]]$.  The affine Grassmannian $ \Gr $ of $ G^\vee $ is defined as $ G^\vee(\K)/G^\vee(\O)$.  The $ G^\vee(\O) $ orbits on $ \Gr $ are labelled by dominant weights of $ G $.  We write $ \Gr^\lambda $ for the $ G^\vee(\O) $ orbit through $ t^\lambda $, for $ \lambda \in \Lambda_+ $. Let $ L_0 = t^0 $ denote the identity coset in $ \Gr $.  These orbits are closed in $ \Gr $ (and hence projective) if and only if $ \lambda $ is a minuscule weight.  More generally, we have that $ \overline{\Gr^\lambda} = \bigcup_{\mu \le \lambda} \Gr^\mu $, where $ \mu  \le \lambda $ means that $ \mu $ is a dominant weight and $ \lambda - \mu $ is a sum of positive roots.  The smooth locus of $ \overline{\Gr^\lambda} $ is exactly $ \Gr^\lambda $.  In particular, $ \overline{\Gr^\lambda} $ is smooth iff $ \lambda $ is minuscule.

Similarly, the $ G^\vee(\K) $ orbits on $ \Gr \times \Gr $ are also labelled by $ \Lambda_+ $ and we write $ L_1 \overset{\lambda}{\edge} L_2 $ if $ (L_1,L_2) $ is in the same orbit as $ (L_0, t^\lambda) $.  

Let $ \underline{\lambda} = (\lambda_1, \dots, \lambda_k) $ be a $ k $ tuple of dominant weights of $ G $.  Then we define the local convolution Grassmannian as
\begin{equation*}
\widetilde{\Gr}^\ul = \{ (L_1, \dots, L_k) \in \Gr^k : L_0 \overset{\lambda_1}{\edge} L_1 \overset{\lambda_2}{\edge} \cdots \overset{\lambda_{k-1}}{\edge} L_{k-1} \overset{\lambda_k}{\edge} L_k \}
\end{equation*}
There is a map $ m_{\ul}: \widetilde{\Gr}^\ul \rightarrow \Gr $ with $ (L_1, \dots, L_k) \mapsto L_k $.  The image of $ m_\ul $ is $ \overline{\Gr^{\lambda_1 + \dots + \lambda_k}}$.

The following result follows from the geometric Satake correspondence \cite{MVi}.
\begin{Theorem} \label{th:geomsat}
Assume that all $ \lambda_i $ are minuscule.  There are canonical isomorphisms 
\begin{enumerate} 
\item $H^*(\Gr^\lambda) \cong V_\lambda$,
\item $H^*(\widetilde{\Gr}^\ul) \cong V_{\lambda_1} \otimes \cdots \otimes V_{\lambda_n}$, and 
\item $H^{\mathrm{top}}(m_{\ul}^{-1}(L_0)) \cong (V_{\lambda_1} \otimes \cdots \otimes V_{\lambda_n})^G$. 
\end{enumerate}
Moreover, the isomorphisms in (ii), (iii) are compatible in the sense that the diagram
\begin{equation*}
\begin{CD}
H^*(\widetilde{\Gr}^\ul) @>\sim>> V_{\lambda_1} \otimes \cdots \otimes V_{\lambda_n} \\
@VVV @VVV \\
H^{\mathrm{top}}(m_{\ul}^{-1}(L_0)) @>\sim>> (V_{\lambda_1} \otimes \cdots \otimes V_{\lambda_n})^G
\end{CD}
\end{equation*}
commutes, where the left vertical map comes from the inclusion $ m_{\ul}^{-1}(L_0) \rightarrow \widetilde{\Gr}^\ul$.
\end{Theorem}

\begin{proof}
The geometric Satake correspodence gives us an equivalence of tensor categories between the category of $ G^\vee(\O) $ equivariant perverse sheaves on $ \Gr $ (called the spherical Hecke category) and the category of representations of $ G $, compatible with the fibre functors to the category of vector spaces \cite[Theorem 7.3]{MVi}.  The fibre functor on the spherical Hecke category is derived global sections.

This equivalence takes $ \C_{\Gr^\lambda}[\dim \Gr^\lambda] $ to $ V_\lambda $ when $ \lambda $ is minuscule.  This immediately gives us (i).  Also from the definition of the tensor product on the spherical Hecke category \cite[Section 4]{MVi}, we see that the equivalence takes $ (m_{\ul})_* \C_{\widetilde{\Gr}^\ul}[\dim \widetilde{\Gr}^\ul] $  to $ V_{\lambda_1} \otimes \cdots \otimes V_{\lambda_n} $.  This gives us (ii) and (iii).  The compatibility statement is clear by construction.
\end{proof}

\subsection{The notion of Hecke type}
Let us now fix a smooth curve $ X $.  Later we will take $ X = \p $.

For any $ x \in X $, and any open set $ U \subset X$ containing $ x $, we define
\begin{equation*}
\begin{aligned}
\Gr_{U, x} := \big\{ (P, \phi) : P &\text{ is a principal } G^\vee \text{ bundle on } U \\ 
&\text{ and } \phi : P_0|_{U \smallsetminus x} \rightarrow P|_{U \smallsetminus x} \text{ is an isomorphism} \big\}.
\end{aligned}
\end{equation*}
here and below, $ P_0 $ denotes the trivial $ G^\vee $ bundle on $ U $.  

Picking a coordinate at $ x $ gives an isomorphism $ \Gr_{U,x} \cong \Gr $.  The isomorphism is well-defined up to the action of $ Aut(\O) $ on $ \Gr $.  Since $\Gr^\lambda $ is $ Aut(\O) $ invariant, we may consider the locus $ \Gr^\lambda_{U,x} $ obtained as the preimage of $ \Gr^\lambda $ under any isomorphism.  An element $ (P, \phi) \in \Gr^\lambda_{U,x} $ is said to have Hecke type $ \lambda$.  Note that $ G^\vee(U) $ acts on $\Gr_{U,x} $ preserving its stratification into Hecke types.

We need to generalize our notion of Hecke type.  Again fix $ x \in X$ and suppose we have two $ G^\vee $ bundles $ (P_1, P_2) $ on $ X $ and an isomorphism $ \phi $ between them over $ X \smallsetminus x $.  Let us pick an open neighbourhood $ U $ of $ x $ on which $ P_1 $ trivializes and pick a trivialization of $P_1 $ on this neighbourhood.  This gives us a point $ L= (P_2|_U, \phi') \in \Gr_{U, x}$ where $ \phi' : P_0|_{U \smallsetminus x} \rightarrow P_2|_{U \smallsetminus x} $ is the composition of this trivialization and $ \phi $.  We say that $ \phi $ has Hecke type $ \lambda $ at $ x $ if $ L  $ has Hecke type $ \lambda $.

Changing the trivialization of $ P_1 $ on $ U $ will change $ L $ by the action of $ G^\vee(U) $ and hence will not change its Hecke type.  Also, changing the open set $ U $ also does not change $ \lambda $, as can be seen by shrinking the open set $ U $.  Hence the Hecke type of $ \phi $ is independent of the choices made in its definition.

For $ X = \p $, there is an alternate characterization of Hecke type.  First, we should note that if $ P $ is a principal $G$-bundle on $ \p$, then we can consider its topological type which will be an element of $ \pi_1(G^\vee) = \Lambda/Q  $ where $ Q $ denotes the coroot lattice of $ G^\vee $.  Using this notion, we have the following result, due to Finkelberg-Mirkovic \cite[Prop 10.2]{FM}.
\begin{Proposition} \label{th:Hecketype}
Let $x, X = \p, \phi, P_1, P_2 $ be as above.  $\phi $ has Hecke type $ \le \lambda $ at $ x $ if and only if 
\begin{enumerate}
\item for all irreps $ V_\beta $ of $ G^\vee $, the map $ \phi $ induces the inclusions
\begin{equation*}
V_\beta^{P_1}(- \langle \beta, \lambda \rangle x) \subset V_\beta^{P_2} \subset V_\beta^{P_1}(\langle \beta, \lambda \rangle x) 
\end{equation*}
(here $V_\beta^P $ denotes the associated vector bundle $ P \times_{G^\vee} V_\beta $) and
\item the topological types of $ P_2 $ and $ P_1 $ differ by $ [\lambda] \in \pi_1(G^\vee) = \Lambda/Q $.
\end{enumerate}
\end{Proposition}
(Note that the second condition here is vacuous when $ G^\vee$ is simply connected.  Also, note that a similar characterization works for any complete curve.  On the other hand, it is not clear to the author what to do with the second condition when considering open curves.)

\subsection{The global convolution Grassmannian}
Now fix $ \ul = (\lam_1, \dots, \lam_n) $ an $n$-tuple of dominant minuscule weights of $ G $.  We now consider the global convolution Grassmannian, which is the variety 
\begin{equation*}
\begin{aligned}
\widetilde{\Gr}^\ul_{X^n} := \{ &((x_1, \dots, x_n), (P_1, \dots, P_n), (\phi_1, \dots, \phi_n)) : x_i \in X, \\ & P_i \text{ is a principal } G^\vee \text{ bundle on } X, \\
& \phi_i : P_{i-1}|_{X \smallsetminus x_i} \rightarrow P_i|_{X \smallsetminus x_i} \text{ is an isomorphism of Hecke type $ \lam_i $ at $ x_i$.} \}
\end{aligned}
\end{equation*}
We have the projection  $ \tilde{p} : \widetilde{\Gr}^\ul_{X^n} \rightarrow X^n $.  For any subset $ A \subset X^n $, let $ \widetilde{\Gr}^\ul_A = \tilde{p}^{-1}(A) $ be the preimage of $ A $ under this map. 

Consider the following two loci in $ X^n$: the locus of regular points $\Xnreg := \{ (x_1, \dots, x_n) : x_i \ne x_j \} $ and the small diagonal $ X = \{ (x,\dots, x) \} \subset X^n  $.  The fibres of $ \tilde{p} $ over regular points in $\Xnreg := \{ (x_1, \dots, x_n) : x_i \ne x_j \} $ are isomorphic to $ \Gr^{\lam_1} \times \dots \times \Gr^{\lam_n} $.  The fibres over points on the small diagonal are the local convolution products $ \Gr^{\lam_1} \ctimes \dots \ctimes \Gr^{\lam_n} $. 

A proof of the following result can be found in the proof of Lemma 6.1 of \cite{MVi}.
\begin{Proposition} \label{th:pushconstant1}
The pushforwards $ R^k\tilde{p}_* \C_{\widetilde{\Gr}^\ul_{X^n}} $ are constant sheaves.
\end{Proposition}

\subsection{The global singular Grassmannian}
Now, we can introduce the singular version of the above space.  We define the global singular Grassmannian to be
\begin{equation*}
\begin{aligned}
\Gr^\ul_{X^n} := \Big\{ \big( &(x_1, \dots, x_n), P, \phi \big) : 
x_i \in X, P \text{ is a principal } G^\vee \text{ bundle on } X, \\
& \phi : P_0|_{X \smallsetminus \{x_1, \dots, x_n \}} \rightarrow P|_{X \smallsetminus \{x_1, \dots, x_n \}} \text{ is an isomorphism }\\
& \text{ and for each $x \in X $, $ \phi $ has Hecke type $ \le \sum_{i \, : \, x_i = x} \lam_i $ at $ x $} \Big\}
\end{aligned}
\end{equation*}

Again this is a family over $X^n$.  
\begin{Proposition}  \label{th:fibres}
The fibres of this family are described as follows.
\begin{enumerate}
\item The fibres over regular points are still $ \Gr^{\lam_1} \times \dots \times \Gr^{\lam_n} $.  In fact $ \widetilde{\Gr}^\ul_{\Xnreg} \cong \Gr^\ul_{\Xnreg} $.
\item The fibres over points on the small diagonal are the (usually singular) varieties $ \overline{\Gr^{\lam_1 + \dots + \lam_n}} $.
\end{enumerate}
\end{Proposition}

There is an obvious map $ \widetilde{\Gr}^\ul_{X^n} \rightarrow \Gr^\ul_{X^n} $.  This map is 1-1 over the locus in $ \Gr^\ul_{X^n} $ which sits over $ \Xnreg $ and coincides with the map $ m_\ul $ on a fibre over a point in the small diagonal in $ X^n $.

The biggest subgroup of $ \Sigma_n $ which acts on $ {\Gr}^\ul_{X^n} $ is denoted $ \Sigma_\ul $.  More precisely, $ \Sigma_\ul $ is the stabilizer of $ \ul $ inside $ \Sigma_n $ (so if all $ \lam_i $ are equal, then $ \Sigma_\ul = \Sigma_n $ and if all $ \lam_i $ are different, then $ \Sigma_\ul = \{1\} $).  

The quotient of $ X^n $ by $ \Sigma_\ul $ is denoted $ X^\ul $.  Since all $ \lambda_i $ are minuscule, distinct $ \lambda_i $ are linearly independent.  Hence we can identify $ X^\ul $ with the space of ``$\Lambda_+$-coloured divisors'' of total weight $ \sum \lam_i $, ie functions $ D : X \rightarrow \Lambda_+ $ such that $ \sum_{x \in \p} D(x) = \sum_i \lam_i $.  

The quotient of $ \Gr^\ul_{X^n} $ by $ \Sigma_\ul $ is denoted $ \Gr^\ul_{X^\ul} $ and we can describe its points as follows.
\begin{equation*}
\begin{aligned}
\Gr^\ul_{X^\ul} := \Big\{ &(D, P, \phi) : 
D \in X^\ul, P \text{ is a principal } G^\vee \text{ bundle on } X, \\
& \phi : P_0|_{\p \smallsetminus \supp(D)} \rightarrow P|_{\p \smallsetminus \supp(D)} \text{ is an isomorphism }\\
& \text{ and for each $ x \in X $, $ \phi $ has Hecke type $ \le D(x)$ at $x $} \Big\}
\end{aligned}
\end{equation*}

Let $ \co_\ul := \Xnreg / \Sigma_\ul $.  We have the smooth family $ \Gr^\ul_{\co_\ul} \rightarrow \co_\ul $ and the following Cartesian square.  
\begin{equation} \label{eq:square}
\begin{CD}
\widetilde{\Gr}^\ul_{\Xnreg} @>>> \Gr^\ul_{\co_\ul} \\
@V\tilde{p}VV @VVpV \\
\Xnreg @>>> \co_\ul 
\end{CD}
\end{equation}

We can consider the monodromy action of $ \pi_1(\co_\ul) $ on the cohomology of fibres in the family $p$.  There is a short exact sequence of groups 
\begin{equation*}
1 \rightarrow \pi_1(\Xnreg) \rightarrow \pi_1(\co_\ul) \rightarrow \Sigma_\ul \rightarrow 1.
\end{equation*}
Since the pushforwards $ R^k \tilde{p}_* \C_{\widetilde{\Gr}^\ul_{X^n}}$ are constant by Proposition \ref{th:pushconstant1}, the monodromy action of $ \pi_1(\Xnreg) $ on the cohomology of the fibres of $ \tilde{p} $ is trivial.  So because of the Cartesian square (\ref{eq:square}), this means that the monodromy action of $ \pi_1(\co_\ul) $ factors through $ \Sigma_\ul $.  Moreover the following fact is true.

\begin{Proposition} \label{th:monod1}
Assume that all $ \lambda_i $ are minuscule and let $ x \in \co_\ul $.  There is an isomorphism $ H^*(\Gr^\ul_x) \cong V_{\lam_1} \otimes \cdots \otimes V_{\lam_n} $, compatible with the actions of $ \Sigma_\ul $ on both sides (by monodromy and by permuting tensor factors).
\end{Proposition}

\begin{proof}
The isomorphism comes from using Proposition \ref{th:monod1}, Proposition \ref{th:fibres}, and Theorem \ref{th:geomsat}.  The compatibility with the actions of $ \Sigma_\ul$ is an immediate consequence of the construction of the commutativity constraint for the spherical Hecke category (see \cite[Section 5]{MVi}). 
\end{proof}

\subsection{The open global convolution/singular Grassmannian}
So far we have dealt with arbitrary curves $ X $.  Now we specialize to $ X = \p$.

Let $ \ul $ be such that $ \lam_1 + \dots + \lam_n $ is in the root lattice of $ G $.  This condition is necessary for $ (V_{\lam_1} \otimes \cdots \otimes V_{\lam_n})^G $ to be non-empty.  

We introduce the open subvariety $ \phantom{}_0\widetilde{\Gr}^\ul_{\p^n} $ of $ \widetilde{\Gr}^\ul_{\p^n}$, which consists of the locus where we impose the constraint that $ P_n $ is isomorphic to the trivial $G^\vee $ bundle on $\p$.  This subvariety would empty if we did not impose the above  condition that $ \lam_1 + \dots + \lam_n $ is in the root lattice.  We call this the open global convolution Grassmannian. 

The significance of this open subvariety is given by the following result.
\begin{Proposition} \label{th:midhom}
\begin{enumerate}
\item For any $ y \in \p$, $ \phantom{}_0\widetilde{\Gr}^\ul_{(y,\dots, y)}$ retracts onto the half-dimensional locus $ m_\ul^{-1}(L_0) $. 
\item There is an isomorphism $ H^\mathrm{mid}(\phantom{}_0\widetilde{\Gr}^\ul_{(y,\dots, y)}) \cong (V_{\lam_1} \otimes \cdots \otimes V_{\lam_n})^G $. 
\end{enumerate}
\end{Proposition}

\begin{proof}
Assume that $ y = 0 \in \p $.  We can identify $ \widetilde{\Gr}^\ul_{(y, \dots, y)} $ with the convolution product $ \Gr^{\lam_1} \ctimes \cdots \ctimes \Gr^{\lam_n} $.  A point $ L $ in $ \Gr $ corresponds to a trivial $ G^\vee $ bundle on $ \p$ if and only if $ L $ is in the $G[z^{-1}] $ orbit of $ L_0 $, which we denote $ \phantom{}_0\Gr$.  So $ \grt_{(y, \dots, y)} $ is identified with $ m_\ul^{-1}(\phantom{}_0\Gr) $.  

The open subset $ \phantom{}_0\Gr $ of the affine Grassmannian retracts onto $ L_0 $ via the loop rotation action of the group $ \C^\times $ (see for example \cite[Section 2]{MVi}).  This action extends to the convolution product and retracts $ m_\ul^{-1}(\phantom{}_0\Gr)  $ onto $ m_\ul^{-1}(L_0) $ as desired. 

We have $ \dim m_\ul^{-1}(L_0) = \frac{1}{2} \dim m_\ul^{-1}(\phantom{}_0 \Gr) $ because the map $ m_\ul $ is semismall (see the proof of Lemma 4.4 in \cite{MVi}).

Hence we have
\begin{equation*}
H^{\mathrm{mid}}(\grt_{(y,\dots, y)}) \cong H^{\mathrm{top}}(m_\ul^{-1}(L_0)) \cong (V_{\lam_1} \otimes \cdots \otimes V_{\lam_n})^G 
\end{equation*}
where the last isomorphism follows from Theorem \ref{th:geomsat}.(iii).
\end{proof}

Analogously, we also define $ \gr_{\p^n} $ and $ \gr_{\p^\ul} $, the open global singular Grassmannian.  So 
\begin{equation*} \gr_{\p^\ul} = \{ (D, P, \phi) \in \Gr^\ul_{\p^\ul} : P \text{ is isomorphic to the trivial bundle } \}  
\end{equation*}
We have a Cartesian square identical to (\ref{eq:square}):
 \begin{equation} \label{eq:square2}
\begin{CD}
\grt_{\Pnreg} @>>> \gr_{\co_\ul} \\
@V\tilde{p}VV @VVpV \\
\Pnreg @>>> \co_\ul 
\end{CD}
\end{equation}

\begin{Proposition} \label{th:constant2}
The pushforwards $ R^k\tilde{p}_* \C_{\grt_{\p^n}} $ are constant sheaves on $ \p^n$.
\end{Proposition}

 We would like to thank Alexander Braverman for suggesting the following proof.
\begin{proof}
Since $ \p^n $ is simply connected, it suffices to show that the push forwards are locally constant.  Since the category of locally constant sheaves is an abelian category, it suffices to find a stratification $ Y_{\um} $ of $ \grt_{\p^n} $ such that each $ R^k\tilde{p}_* \C_{Y_\um} $ is locally constant.

Recall that isomorphism classes of $ G^\vee $ bundles on $ \p $ are given by dominant weights $ \mu $ of $ G $ (this is because every $ G^\vee $ bundle on $ \p $ admits a reduction to $ T^\vee $).  Hence for $ \um = (\mu_1, \dots, \mu_n = 0) $, we define
\begin{equation*}
Y_{\um} := \{ (P_1,\dots, P_n) \in \grt_{\p^n} : P_i \text{ has isomorphism type } \mu_i \text{ for } i = 1, \dots, n \}.
\end{equation*}

The family $ Y_\um \rightarrow \p^n $ is trivial and thus the pushforwards $ R^k \tilde{p}_* \C_{Y_\um} $ are constant sheaves.  Hence the result follows.
\end{proof}

By similar reasoning as in the paragraph before Proposition \ref{th:monod1}, we can use Proposition \ref{th:constant2} to see that the monodromy action of $\pi_1(\co_\ul)$ on the cohomology of the fibres of $p : \gr_{\co_\ul} \rightarrow \co_\ul $ factors through $ \Sigma_\ul $.   
We obtain the following description of the monodromy action on the middle homology  of the fibres of $ \gr_{\co_\ul} $.
\begin{Proposition}  \label{th:monod2}
Let $ x \in \Pnreg $.  There is an isomorphism $ H_{\mathrm{mid}}(\gr_x) \cong (V_{\lam_1} \otimes \cdots \otimes V_{\lam_n})^G $, which is compatible with the actions of $ \Sigma_\ul $ on both sides.
\end{Proposition}

\begin{proof}
Using Proposition \ref{th:constant2} and Proposition \ref{th:midhom}, we obtain the isomorphism 
\begin{equation*} H_{\mathrm{mid}}(\gr_x) = H_{\mathrm{mid}}(\grt_x) \cong H_{\mathrm{mid}}(\grt_{(y, \dots, y)})
\end{equation*}
where $ y \in \p $.  

By the last statement of Proposition \ref{th:geomsat} we  obtain a commutative square
\begin{equation*}
\begin{CD}
H^*(\Gr^\ul_x) @>\sim>> V_{\lambda_1} \otimes \cdots \otimes V_{\lambda_n} \\
@VVV @VVV \\
H^{mid}(\gr_x) @>\sim>> (V_{\lambda_1} \otimes \cdots \otimes V_{\lambda_n})^G
\end{CD}
\end{equation*}
By Proposition \ref{th:monod1}, the top horizontal arrow is $ \Sigma_\ul $ equivariant.  The right vertical arrow is clearly $ \Sigma_\ul $ equivariant and the left vertical arrow is clearly $ \pi_1(\co_\ul) $ equivariant.  Hence we conclude that the bottom horizontal arrow is $ \Sigma_\ul $ equivariant as desired.
\end{proof}

\begin{Remark}
It would be interesting to generalize the construction of $\gr_{\p^n}$ to arbitrary curves $ X $.  To do so, one needs to impose a condition on a bundle $ P $.  The two natural choices are $ P $ is trivial and $P $ is semistable.  Imposing that $ P $ is trivial would give us too small a space and imposing that $ P $ is semistable seems to give too big a space (I believe we lose the fact that the fibres are affine).
\end{Remark} 

\section{Semilocal geometry}
The purpose of this section is to establish the behaviour of our fibration when two points come together.  The basic tool is the factorization property of the Grassmannians.  We begin by stating this property which is due to Beilinson-Drinfeld.

\begin{Theorem}\cite[section 5.3.10]{BD}
Let $ X $ be any smooth curve.  Let $ J_1, J_2 $ be a disjoint decomposition of $ (1, \dots, n) $.  Let $ U $ denote the set of points $[x_1, \dots, x_n ] $ in $ X^\ul $ such that $ x_{j_1} \ne x_{j_2} $ if $ j_1 \in J_1 $ and $ j_2 \in J_2 $.  Then there is an isomorphism
\begin{equation*}
\Gr^\ul_U \cong (\Gr^{\ul^1}_{X^{J_1}} \times \Gr^{\ul^2}_{X^{J_2}}) |_U 
\end{equation*}
where $ \ul^i $ consists of those $ \lambda_j $ with $ j \in J_i $.  This isomorphism is compatible with the two projections to $ U $.
\end{Theorem}

Note that we have already implicitly used a special case of this Theorem in Proposition \ref{th:fibres}.(i).

\subsection{Semilocal geometry in the global singular Grassmannian}
For $ \lambda \in X_+ $, let $ \lambda^\vee = -w_0(\lambda) $, where $ w_0 $ is the long element of the Weyl group.  Equivalently, we have $ V^{\lambda^\vee} = (V^\lambda)^\vee $.

Let us assume that $ \mu = \lam_i = \lam_{i+1}^\vee $.  Consider $ x = [x_1, \dots, x_n] \in \p^\ul $ with $y =  x_i = x_{i+1} $ and $ x_j \ne x_k $ otherwise.  By the factorization property of the Grassmannians, we have that $ \Gr_x^\ul \cong \Gr^{d_i(\ul)}_{d_i(x)} \times \Gr^{(\mu, \mu^\vee)}_{(y,y)} $ where $ d_i $ deletes the $ i $ and $ i+1$ entries of a list. 

The variety $\Gr^{(\mu, \mu^\vee)}_{(y,y)} $ is isomorphic to $ \overline{\Gr^{\mu + \mu^\vee}} $.  It is singular with a stratum for each $ \nu $ such that $ V_\nu $ appears in $ V_\mu \otimes V_\mu^\vee $.  In particular, we have a distinguished point which corresponds to the trivial bundle with trivial isomorphism.  This gives a copy of $\Gr^{d_i(\ul)}_{d_i(x)}$  embedded in $\Gr_x^\ul $ as the locus of $ (D, P, \phi) $ where the isomorphism $ \phi $ extends over $ y $ (equivalently, has Hecke type $0 $ at $y $).  

\subsection{Semilocal geometry in the open singular Grassmannian}
Now, let us pass to the open singular Grassmannian (the main object of study).  The singular fibre $ \gr_x $ does not factor as a product in this case.  However, we do have a copy of $ \phantom{}_0 \Gr^{d_i(\ul)}_{d_i(x)} $ embedded  in $ \gr_x $ as the locus of $ (D, P, \phi) $ where $ \phi $ extends over $ y$.  More generally we have the following ``semilocal geometry'' which matches Lemma 3.5 of \cite{M} and Lemma 21 of \cite{SS}.

Let $ r $ be chosen such that $ r \le |x_j - y| $ for all $ j $.  Let $ B \subset \p^\ul $ be a disc corresponding to the points of the form $ [x_1, \dots, y - u, y + u, \dots, x_n] $, where $ |u | < r$ and let $ B' \subset (\p)^{(\mu, \mu^\vee)} $ be the points of the form $ [y-u, y+u] $, for $ |u | < r $.

\begin{Lemma} \label{th:semilocal}
There exists an open neighbourhood of $ \phantom{}_0 \Gr^{d_i(\ul)}_{d_i(x)} $ in $ \gr_B $ and an isomorphism $ \psi $ of this neighbourhood with a neighbourhood of $ \phantom{}_0 \Gr^{d_i(\ul)}_{d_i(x)}  $ in $ \phantom{}_0 \Gr^{d_i(\ul)}_{d_i(x)} \times \phantom{}_0 \Gr^{(\mu, \mu^\vee)}_{B'} $ such that the following diagram commutes
\begin{equation*}
\begin{CD}
\gr_B @>\psi>> \phantom{}_0 \Gr^{d_i(\ul)}_{d_i(x)} \times \phantom{}_0 \Gr^{(\mu, \mu^\vee)}_{B'} \\
@VVV @VVV \\
B @>\sim>> B'
\end{CD}
\end{equation*}
\end{Lemma}

In \cite{SS} and \cite{M}, the corresponding Lemma is used to construct a Lagrangian $M' $ in $\gr_x $ from a Lagrangian $ M $ in $ \phantom{}_0 \Gr_{d_i(x)}^{d_i(\ul)} $ via a relative vanishing cycle construction (see \cite[section 4.4]{M}).  In these cases the Lagrangian $M'$ was diffeomorphic to $ \mathbb{P}^{m-1} \times M $.  In our case, we expect that a similar construction will exist with $ \mathbb{P}^{m-1} $ replaced by $ m_{(\mu, \mu^\vee)}^{-1}(L_0) = \Gr^\mu $, which is a cominuscule partial flag variety for the group $ G^\vee $.  

This construction can be interpreted as a functor from the derived Fukaya category of $ \phantom{}_0 \Gr_{d_i(x)}^{d_i(\ul)} $ to that of $ \gr_x $ (see \cite{R}).  This functor should categorify the map
\begin{equation*}
\phantom{}_0 V^{d_i(\ul)} = (V^{\lambda_1} \otimes \cdots \otimes V^{\lambda_n})^G \rightarrow (V^{\lambda_1} \otimes \cdots \otimes V^\mu \otimes V^{\mu^\vee} \otimes \cdots \otimes V^{\lambda_n})^G = \phantom{}_0V^{\ul}.
\end{equation*}

\begin{proof}
By the factorization property, we have, as above, an isomorphism
\begin{equation*}
\psi : \Gr_B^\ul \cong \Gr^{d_i(\ul)}_{d_i(x)} \times \Gr^{(\mu, \mu^\vee)}_{B'}.
\end{equation*}
Now, we let $ V \subset \Gr_B^\ul $ be defined as the intersection of $ \gr_B $ with $ \psi^{-1} \big( \phantom{}_0 \Gr^{d_i(\ul)}_{d_i(x)} \times \phantom{}_0 \Gr^{(\lambda, \lambda^\vee)}_{B'} \big) $.  By construction $ V $ is a neighbourhood of $\phantom{}_0 \Gr^{d_i(\ul)}_{d_i(x)} $ on each side.  The restriction of $ \psi $ to $ V $ fulfills the hypotheses.
\end{proof}

\section{Comparison with resolution of slices}
\subsection{Slices and their resolutions}
Fix a positive integer $ N = mk$.  We will study slices inside $ \mathfrak{gl}_N $.  

Consider the nilpotent matrix $ E_{m,k} $ which consists of $ k -1 $ copies of the $m\times m $ identity matrix arranged below the diagonal.  
\begin{equation*}
\begin{pmatrix}
0 & 0 & 0 & 0 \\
I & 0 & 0 & 0 \\
0 & I & 0 & 0 \\
0 & 0 & I & 0 
\end{pmatrix}
\end{equation*}
This is the matrix for the linear operator $ z $ acting on $ \C[z]^m / z^k \C[z]^m $ with respect to the usual basis $ e_1, \dots, e_m, z e_1, \dots, z^{k-1} e_m $.

Let $ F_{m,k} $ be the matrix completing it to a Jacobson-Morozov triple.  Let $ S_{m,k} = E_{m,k} + ker( \cdot F_{m,k}) $.  So $ S_{m,k}$ is the set of block matrices consisting of identity matrices below the diagonal and arbitrary matrices on the right column.  In particular, it is an affine space.
\begin{equation*}
\begin{pmatrix}
0 & 0 & 0 & * \\
I & 0 & 0 & * \\
0 & I & 0 & * \\
0 & 0 & I & * 
\end{pmatrix}
\end{equation*}
Here each block is of size $ m \times m $.

Let $ \pi = (\pi_1, \dots, \pi_n) $ a sequence of integers with $ 1 \le \pi_i \le m-1 $ and $ \pi_1 + \dots + \pi_n = N $.  We can consider the partial flag variety $ \fl_\pi $ of type $ \pi $ (that means that the jumps are given by $ \pi $).  We now define the partial Grothendieck resolution $ \tg^\pi \subset \fl_\pi \times \mathfrak{gl}_N $ by
\begin{equation*}
\tg^\pi := \{ (Y, W_\bullet) : Y W_i \subset W_i \text{ and $ Y $ acts as a scalar on } W_i/W_{i-1} \}.
\end{equation*}
By recording the scalars with which $ Y $ acts on each successive quotient, we obtain a morphism $ \tg^\pi \rightarrow \mathbb{C}^n $.

Of course there is also a morphism $ \tg^\pi \rightarrow \mathfrak{gl}_N $.  Let $ \widetilde{S}_{m,k}^\pi $ denote the preimage of $ S_{m,k} $ inside $\tg^\pi $.

Let $ \mathfrak{gl}_N^\pi $ denote the image of $ \tg^\pi $ inside $ \mathfrak{gl}_N $. It consists of those matrices such that $\pi $ refines their partition of eigenvalues.  In particular any matrix in $ \mathfrak{gl}_N^\pi $ has at most $ n $ distinct eigenvalues.

Let $ \Sigma_\pi $ denote the stabilizer of $ \pi $ in $ \Sigma_n $ and let $ \mathbb{C}^\pi := \C^n/\Sigma_\pi $.  Let $ S_{m,k}^\pi $ denote the following variety.
\begin{equation*}
\begin{aligned}
S_{m,k}^\pi := \big\{ (Y, [x_1, &\dots, x_n]) : Y \in S_{m,k} \cap \mathfrak{gl}_N^\pi, [x_1, \dots, x_n] \in \C^\pi \\ 
&\text{ and } \{ x_1 \}^{\cup \pi_1} \cup \dots \cup \{x_n \}^{\cup \pi_n} \text{ are the eigenvalues of } Y \big\}
\end{aligned}
\end{equation*}

\subsection{Examples and relation with \cite{M}, \cite{SS}}
Consider the case $ m= 2$.  This forces $ \pi = (1, \dots, 1) $ and $ n=N $.  So $ \tg^\pi = \tg $ is the ordinary Grothendieck resolution.  This places no restriction on the matrices, so $ \mathfrak{gl}_N^\pi = \mathfrak{gl}_N $.  Also $ \Sigma_\pi = \Sigma_n $ and so $ \C^\pi = \C^n/\Sigma_n $.  We also see that $ S_{m,k}^\pi = S_{m,k} $.  This is precisely the slice considered by Seidel-Smith which maps down to the space of eigenvalues $ \C^n/\Sigma_n $.

Now, consider the case $ m $ arbitrary and $ \pi = (1, \dots, 1, m-1, \dots, m-1) $ where there are $ k$ 1s and $ k$ m-1s.  Then $ \Sigma_\pi = \Sigma_k \times \Sigma_k $ and $ \C^\pi = \C^{2k} / \Sigma_k \times \Sigma_k $ is the space of collections of ``thin'' and ``thick'' eigenvalues, to use the terminology of Manolescu.  The regular locus is the bipartite configuration space $BConf_k$ studied by Manolescu.  The fibres of $ S^\pi_{m,k} \rightarrow BConf_k $ is the fibration studied by Manolescu.  

Manolescu also develops the geometry in the same general framework that we do.  To continue the comparison of terminology, our $\pi $ is his $ \pi $, our $ \tg^\pi $ is his $ \mathfrak{g}^\pi $, our $\mathfrak{gl}_N^\pi $ is his $ \mathfrak{g}^\pi $, our $ \Sigma_\pi $ is his $ W^\pi $, our $\C^n $ is his $ \mathfrak{h}^\pi = \C^{s-1} $, our $ \C^\pi $ is his $ \mathfrak{h}^\pi/W^\pi $.  We are unable to reconcile his map $ \mathfrak{g}^\pi \rightarrow \mathfrak{h}^\pi/W^\pi $.  We don't believe that such a map exists and that is why we constructed the more complicated $ S_{m,k}^\pi $.

\subsection{Comparison theorem}
Now we take $ G = SL_m $, so $ G^\vee = PGL_m$.  We have the usual labelling $ \omega_1, \dots, \omega_{m-1} $ of minuscule weights for $ SL_m $.

Let $ \ul = (\lam_1, \dots, \lam_n) $ be such that $ \lam_1 + \dots + \lam_n $ is in the root lattice and each $ \lambda_i $ is minuscule.  For each $ i $, let $ \pi_i \in \{1, \dots, m-1 \} $ be such that $ \lam_i = \omega_{\pi_i} $.  Then, $ N := \pi_1 + \dots + \pi_n $ is divisible by $ m $ (in order so that $ \lam_1 + \dots + \lam_n $ is in the root lattice), so $ N = mk $ for some positive integer $ k $.  Note that $ \Sigma_\ul = \Sigma_\pi $ and $ \mathbb{C}^\ul = \mathbb{C}^\pi $ (here $ \C^\ul $ denotes the subvariety of $ \p^\ul $ where all points are in $ \C $).

The following result is due to Mirkovic-Vybornov \cite[Theorem 5.3]{MVy} and Ngo \cite[Lemma 2.3.1]{N}.

\begin{Theorem} \label{th:comp}
There exist isomorphisms $ \grt_{\mathbb{C}^n} \cong \widetilde{S}_{m,k}^\pi $ and $ \gr_{\mathbb{C}^\ul} \cong S_{m,k}^\pi $ which are compatible with the following two squares.
\begin{equation*}
\begin{CD}
\grt_{\mathbb{C}^n} @>>> \gr_{\mathbb{C}^\ul} \\
@VVV @VVV \\
\mathbb{C}^n  @>>> \mathbb{C}^\ul
\end{CD}
\quad \quad  \quad
\begin{CD}
\widetilde{S}_{m,k}^\pi  @>>> S_{m,k}^\pi \\
@VVV @VVV \\
\mathbb{C}^n  @>>> \mathbb{C}^\pi
\end{CD}
\end{equation*}
\end{Theorem}

Though the theorem has already been proved by Ngo and Mirkovic-Vybornov, we provide an sketch proof for the convenience of the reader.

\begin{proof}
We will show that both $ \grt_{\mathbb{C}^n} $ and $ \widetilde{S}_{m,k}^\pi $ are isomorphic to the following space.
\begin{equation*}
\begin{aligned}
\phantom{}_0\mathcal{L}^\ul := \big\{ &L_0 = \C[z]^m \supset L_1 \supset \dots \supset L_n: L_i \text{ are free } \C[z] \text{ submodules of rank } m, \\
&\dim(L_{i-1}/L_i) = \pi_i, z \text{ acts by a scalar on } L_{i-1}/L_i,\text{ and}  \\
& [e_1], \dots [e_m], [ze_1], \dots [ze_m], \dots, [z^{k-1}e_1], \dots, [z^{k-1}e_m] \text{ is a basis} \text{ for } \C[z]^m/L_n  \big\}
\end{aligned}
\end{equation*}

First, suppose $ P_1, P_2 $ are two principal $ PGL_m $ bundles on $\p$ and $ \phi $ is an isomorphism between them away from a point $ x \ne \infty$ of Hecke type $ \omega_j $.  Then there exist rank $ m$ vector bundles $V_1, V_2 $ representing $ P_1, P_2 $ such that $ L_2 := \Gamma(\C, V_2) \subset L_1 := \Gamma(\C, V_1) $ and $ \dim (L_1/L_2) = j $.  Moreover, we see that $ (z-x)L_1 \subset L_2 $.  (This is the $ PGL_m $ version of Proposition \ref{th:Hecketype}.)

This allows us to define an isomorphism $ \widetilde{\Gr}^\ul_{\C^n} \rightarrow \mathcal{L}^\ul $, where $ \mathcal{L}^\ul $ is just like $ \phantom{}_0\mathcal{L}^\ul $ except without the last condition.

Hence it suffices to show that if a vector bundle $ V $ on $ \p $ has space of sections $ L= \Gamma(\C, V) \subset \C[z]^m $, then $ [e_1], \dots, [z^{k-1} e_m] $ is a basis if and only if $ V \cong \mathcal{O}(k)^{\oplus m} $ (this latter condition implies that $ \mathbb{P}(V) $ is trivial and in this case it is in fact equivalent to it).

To prove this claim, let us define an isomorphism $ V \rightarrow \mathcal{O}(k)^{\oplus m} $.  To do so it suffices to define an isomorphism $ \psi $ of $ \C[z] $ modules $ \psi: L \rightarrow z^k \C[z]^{\oplus m} $ which extends over $ \p$.  Let $ B $ be the set $ \{ e_1, \dots, z^{k-1} e_m \} $. For each $ i $ define $ p_i $ to be the unique element of $ \spn(B) $ such that $ p_i - z^k e_i \in L $ (such a $ p_i $ exists by the ``basis for quotient'' assumption).  Note that $ \{ z^k e_1 - p_1, \dots, z^ke_m - p_m \} $ forms a basis for $ L $ as a free $ \C[z] $ module.  Now, we define $ \psi $ to take $ z^k e_i - p_i $ to $ z^k e_i $.  Since it takes one basis to another, $ \psi $ is a isomorphism of $ \C[z] $ modules.  Finally, a simple calculation shows that $ \psi $ extends to a isomorphism over $ \p $.

Thus, we have proven that $ \grt_{\C^n} $ is isomorphic to $ \phantom{}_0\mathcal{L}^\ul $.  So it remains to construct the isomorphism between $ \phantom{}_0\mathcal{L}^\ul $ and $ \widetilde{S}_{m,k}^\pi $.

Given $ (L_0, \dots, L_n) \in \mathcal{L}^\ul $, we consider the linear operator $ Y $ defined as the action of $ z $ on the quotient $ L_0/L_n $.  Since we have a preferred basis for $ L_0/L_n $, we may identify $ L_0/L_n $ with $ \C^N $.  The sequence $ L_n/L_n, L_{n-1}/L_n, \dots, L_0/L_n $ gives us a flag $ W_\bullet $ of type $ \pi $ in $ \C^N $ such that $ Y $ preserves each subspace and acts as a scalar on each quotient.  Thus $ (Y, W_\bullet) $ is a point in $ \tg^\pi $.  Moreover, we see that $ Y [z^i e_j] = [z^{i+1}e_j] $ for each $ 0 \le i \le k-2 $ and all $j $.  Thus $ Y \in S_{m,k} $ as desired.  (Note that the condition that $ L_n = z^k \C[z]^m $ is equivalent to the condition $ Y = E_{m,k}$.)  Hence $ (Y, W_\bullet) $ lies in $ \widetilde{S}_{m,k}^\pi $.  An inverse map can easily be constructed by an explicit formula (see \cite[section 4.5]{MVy}).
\end{proof}

\section{Affine nature of the fibres}
As a final step, we would like to prove the following result.

\begin{Proposition}
Let $ x \in \Pnreg $.  Then $ \gr_x $ is an affine variety.
\end{Proposition}

We do not know a truly satisfactory proof of this result.  Here is a slightly ad-hoc approach which uses the results of the previous section. 

\begin{proof}
We begin with the case of $ G^\vee = PSL_m $.  We may assume that $ x = (x_1, \dots, x_n) \in \Cnreg $, since the fibre does not depend on $ x $.  Then by Theorem \ref{th:comp}, $ \gr_x $ is the variety of matrices in $ S_{m,k} $ whose set of eigenvalues is $ \{ x_1 \}^{\cup \pi_1} \cup \dots \cup \{ x_n \}^{\cup \pi_n} $.  Since $ S_{m,k} $ is an affine space and imposing this set of eigenvalues is a closed condition, we see that $ \gr_x $ is an affine variety.  The varieties $ \gr_x $ for $ G^\vee = GL_m $ are the same as those for $ G^\vee = PSL_m $, so this establishes that case too.

Now suppose that $ G^\vee $ is arbitrary.  Pick a faithful representation $ \rho : G^\vee \rightarrow GL_m $ which takes the maximal torus $ T $ of $ G^\vee $ into the maximal torus $T_m $.  This gives us a map $ \rho : Bun_{G^\vee}(\p) \rightarrow Bun_{GL_m}(\p) $ and a closed embedding $ \rho : \Gr_{G^\vee,x}^\ul \rightarrow \Gr_{GL_m,x}^{\rho(\ul)} $ (here $ \rho(\ul) $ denotes the result of transforming $ \ul $ into a tuple of dominant coweights for $ GL_m $ via the map $ T \rightarrow T_m $).  By Lemma \ref{th:triv} below, we see that $ \rho^{-1}(\phantom{}_0\Gr_{GL_m, x}^{\rho(\ul)}) = \gr_{G^\vee, x} $ and hence $ \gr_{G^\vee, x} $ is a closed subvariety of $ \phantom{}_0\Gr_{GL_m, x}^{\rho(\ul)} $ and hence is affine.
\end{proof}

\begin{Lemma} \label{th:triv}
Let $ \rho : G^\vee \rightarrow GL_m $ be as above.  Let $ P $ be a principal $ G^\vee $ bundle on $ \p $.  Then $ P$ is trivial if and only if $ \rho(P) $ is trivial.
\end{Lemma}

\begin{proof}
By Grothendieck's theorem, every $ G^\vee $ bundle on $ \p $ admits a reduction to $ T$.  Hence our bundle $ P $ comes from a $ T $ bundle $ P' $.  Note that $ P $ is trivial iff $ P' $ is trivial.   We may first turn $P' $ into a $ T_m $ bundle $ \rho_T(P') $ and then into the $ GL_m $ bundle $\rho(P) $.

Now, $ T$ bundles are determined by maps $ X(T) \rightarrow \Z $.  The bundle $ \rho_T(P') $ is determined by the transformed map $ X(T_m) \rightarrow X(T) \rightarrow \Z $.  Since $ T \rightarrow T_m $ is an embedding, $ X(T_m) \rightarrow X(T) $ is onto.

Hence we deduce
\begin{equation*}
\begin{aligned}
P \text{ is trivial } &\Longleftrightarrow P' \text{ is trivial } \Longleftrightarrow \text{ the map $ X(T) \rightarrow \Z $ is zero } \\
&\Longleftrightarrow \text{ the map $ X(T_m) \rightarrow X(T) \rightarrow \Z $ is zero } \\ &\Longleftrightarrow \rho_T(P') \text{ is trivial} \Longleftrightarrow \rho(P) \text{ is trivial}.
\end{aligned}
\end{equation*}
\end{proof}

\end{document}